\documentclass[12pt]{article}
\usepackage{amsmath}
\usepackage{amssymb}
\usepackage{amsthm}
\usepackage{extarrows}
\usepackage{graphicx}
\usepackage{amssymb}
\usepackage[retainorgcmds]{IEEEtrantools}
\usepackage{mathrsfs}

\newtheorem{theorem}{Theorem}

\newtheorem{definition}{Definition}

\newtheorem{proposition}{Proposition}

\newcommand{\C}{\mathbb{C}}
\newcommand{\R}{\mathbb{R}}

\begin{document}

\title{Hermitian Clifford Analysis and Its Connections with Representation Theory}

\author{Stuart Shirrell\thanks{Electronic address:  {\tt stuart.shirrell@gmail.com} } $^{1,2}$ and Raymond Walter\thanks{Electronic address:  {\tt rwalter@email.uark.edu} } $^{3,4}$ \\
\emph{\small $^1$Yale Law School, New Haven, CT 06511} \\
\emph{\small $^2$IDinsight, Patna, Bihar, India} \\
\emph{\small $^3$Department of Mathematics, University of Arkansas, Fayetteville, AR 72701, USA} \\ 
\emph{\small $^4$Department of Physics, University of Arkansas, Fayetteville, AR 72701, USA}}

\date{}

\maketitle

\begin{abstract}
This work reconsiders the holomorphic and anti-holomorphic Dirac operators of Hermitian Clifford analysis to determine whether or not they are the natural generalization of the orthogonal Dirac operator to spaces with complex structure. We argue the generalized gradient construction of Stein and Weiss based on representation theory of Lie groups is the natural way to construct such a Dirac-type operator because applied to a Riemannian spin manifold it provides the Atiyah-Singer Dirac operator. This method, however, does not apply to these Hermitian Dirac operators because the representations of the unitary group used are not irreducible, causing problems in considering invariance under a group larger than $U(n)$. This motivates either the development of Clifford analysis over a complex vector space with respect to a Hermitian inner product or the development of Dirac-type operators on Cauchy-Riemann structures.
\\\\ The authors dedicate this work to John Ryan in celebration of his sixtieth birthday.
\end{abstract}


%{\bf AMS Subject Classification: 11E88, 17B10, 57R15} 

\section{Introduction}

The orthogonal Dirac operator has an important interpretation as a generalized gradient operator, constructed in terms of representations of the spin group.  In particular, the orthogonal Dirac operator can be constructed by projecting the gradient onto the orthogonal complement of the Cartan product of appropriate representations of the spin group.  One can generalize this construction to a principal $Spin(n)$ bundle over an oriented Riemannian manifold, giving one a connection on the tangent bundle taking values in a spin bundle.  We thus consider this the correct way to introduce a Dirac-type operator, that is, through representation-theoretic considerations.  A recent development in the field of Clifford analysis is the introduction of two Hermitian Dirac operators, which are, essentially, the holomorphic and anti-holomorphic parts of the Dirac operator.  These operators have been introduced in a way that is similar to the representation-theoretic construction of the Dirac operator.  We show, however, that these operators are not constructed using \emph{irreducible} representations, which creates problems when trying to prove invariance of these operators under a larger group than $U(n)$. These two Hermitian Dirac operators have since been shown to lack invariance under Kelvin inversion \cite{RCK}, reaching the same conclusions, but the present work proceeds strictly using representation theory.

\section{Preliminaries on Orthogonal Clifford Analysis}

Here we review the construction of a Clifford algebra on $\R^n$ and the orthogonal Dirac operator acting on functions taking values on such a Clifford analysis. Consider a real vector space $V$ which is equipped with a non-degenerate, symmetric bilinear form $(\cdot,\cdot)$.  The Clifford algebra associated to $(V,(\cdot,\cdot))$ is defined as follows.  Set

\begin{equation} \nonumber
T(V) = \bigoplus_{i=0}^\infty V^{\otimes i}
\end{equation}
where $V^{\otimes i} = \underbrace{V \otimes V \otimes \cdots \otimes V}_\text{$i$ times}$ and $V^0 = \R$.  Let $I$ be the ideal in $T(V)$ generated by elements of the form $v\otimes w + w\otimes v + 2(v,w)$, where $v,w\in V$.  The Clifford algebra associated to $\left(V, (\cdot,\cdot) \right)$ is then defined by $Cl(V, (\cdot,\cdot))= T(V)/I$. Two examples include the complex numbers and the quaternions.
\begin{enumerate}
\item If $V=\R$ and $(\cdot,\cdot)$ is the standard inner product on $\R$, then we have the associated Clifford algebra being defined by the relations $e_1e_1 + e_1 e_1 = -2$, that is $2e_1^2 =-2$.  Now define a map $Cl(\R) \rightarrow \C$ by $e_1 \mapsto i$, and we see that $Cl(\R) \cong \C$. \\
 \item Consider $V = \R^2$ with the standard inner product.  Then the associated Clifford algebra consists of $\text{span}_\R(1, e_1, e_2, e_1e_2)$ such that $e_1^2 = e_1^2 = (e_1e_2)^2 = -1$ and $e_1e_2 = -e_2e_1$.  Note that we can define a map from this Clifford algebra to the quaternions $\mathbb{H}$ by $e_1 \mapsto \mathbf{i}, e_2 \mapsto \mathbf{j}, e_1e_2 \mapsto \mathbf{k}$.  This map is an isomorphism of division algebras.
\end{enumerate}

These examples show the Clifford algebra is a natural object to consider and moreover that the algebraic structure is intrinsically linked with the geometric structure given by the symmetric bilinear form on the vector space.  We will primarily consider the Clifford algebra associated to $\R^n$ with its standard inner product, which we denote by $Cl_n$.

We now introduce the (orthogonal) Dirac operator on $\R^n= \text{span}_\R (\mathbf{e}_1, \ldots, \mathbf{e}_n)$, where the $\mathbf{e}_i$ form an orthonormal basis.  We define

\begin{equation}
D = \sum_{i=1}^n \mathbf{e}_i \frac{\partial}{\partial x_i}. \nonumber
\end{equation}

This operator acts on functions defined on open subsets of $\R^n$ taking values in the Clifford algebra $Cl_n$.  Treating the $\frac{\partial}{\partial x_i}$ as a scalar, we note that we have the equation $D^2 = - \triangle$, where $\triangle = \sum_{i=1}^n \frac{\partial^2}{\partial x_i^2}$ is the Laplacian in Euclidean space.  This follows from the fact $\mathbf{x}^2 = - \| \mathbf{x}\|^2$ for $\mathbf{x}\in \R^n$.  We say a function $f$ is left (resp. right) monogenic if $Df= 0$ (resp. $fD=0$, where $fD = \sum_{i=1}^n \frac{\partial f}{\partial x_i} \mathbf{e}_i$).   These functions are the Clifford analysis analogues of holomorphic functions in one complex variable. Indeed, there is a well developed function theory of the orthogonal Dirac operator, involving such fundamental results as a Cauchy Theorem, Cauchy integral formula, and analyticity of monogenic functions \cite{BDS, De, P1, R, R1}.  Note, furthermore, that since $D^2 = -\triangle$, every left (or right) monogenic function is also harmonic.  Hence, orthogonal Clifford analysis not only generalizes classical complex analysis but also refines harmonic analysis on $\R^n$ \cite{G}.  

\section{Representation Theory and the Dirac Operator}

This section reviews that representation theory sufficient for the generalized gradient construction of Stein and Weiss \cite{Branson, Fegan, G, Parabolic, SW}, which derives first order differential operators from Lie group representations. We then review the construction of the orthogonal Dirac operator using representation theory for the spin group, emphasizing that this construction proceeds using \emph{irreducible} representations of the spin group, following the argument of \cite{SW}. The construction generalizes further to representations of principal bundles over oriented Riemannian manifolds and, since the Atiyah-Singer Dirac operator is so constructed on a Riemannian spin manifold, we argue the Stein and Weiss construction is the natural way to construct other Dirac-type operators \cite{G}. The next section will consider whether the Hermitian Dirac operators fit in this natural framework.

\subsection{The General Theory}%%%%%
We begin by defining our objects of interest, Lie groups and their representations.
\begin{definition} 
\text{  } 
\begin{enumerate}
\item A \emph{Lie group} is a smooth manifold $G$ that is also a group such that multiplication $(g,h)\mapsto gh:G \times G \rightarrow G$ and inversion $g \mapsto g^{-1}:G \rightarrow G$ are both smooth.
\item A \emph{representation} of a Lie group is a smooth homomorphism $\rho : G \rightarrow GL(V)$ for some real or complex vector space V.
\item A subspace $W\subset V$ for a representation $\rho$ on V is said to be \emph{invariant} if $\rho(g)W \subset W$ for all $g\in G$. 
\item An \emph{irreducible} representation has no nontrivial subspaces, that is, $\{0\}$ and $V$ are the only invariant subspaces.
\end{enumerate}
\end{definition}

We only consider finite-dimensional representations of compact Lie groups.  A basic theorem of Lie theory is that all representations of compact Lie groups are finite dimensional and that any finite-dimensional representation of a compact Lie group decomposes as a direct sum of irreducible representations \cite{knapp}.  A representation maps a Lie group into the general linear group of some vector space.  Hence, there is a well defined trace which is invariant under conjugation by invertible linear transformations; indeed, the trace only needs to be defined on the conjugacy classes of the group in question.  

\begin{definition}
The character of a representation of a Lie group $G$ is defined to be the trace of the representation, that is for a representation $\rho$ of G, the character, $\chi_\rho$, of $\rho$ is given by $\chi_\rho(g) = \text{Tr}(\rho(g))$, $g\in G$. 
\end{definition}
Every compact Lie group has a unique left-invariant measure $\mu$, known as Haar measure, such that $\mu(G)=1$ \cite{knapp}.  The characters corresponding to the finite-dimensional, unitary representation of a finite-dimensional Lie group $G$ determine the representation: they form an orthonormal basis for $\mathscr{L}^2(G,\mu)^G$, the space of square-integrable class functions (that is, functions which are constant on each conjugacy class). This is the so-called Peter-Weyl Theorem and more information on analysis on Lie groups is found in \cite{knapp}.

Given a compact Lie group $G$, we can consider the closed, connected, abelian subgroups of $G$.  Each such group is a direct product of copies of $S^1$.  A maximal abelian subgroup of $G$ is a \emph{maximal torus} of $G$ and is denoted by $T$.  The maximal torus has the form $T= \{(e^{i\theta_1},\cdots, e^{i\theta_k})|\ \theta_i \in \R \ \forall i\}$.  Since $T$ is abelian, each irreducible complex representation of $T$ must be one-dimensional, and hence has the form
\begin{equation}
\label{torus}
(e^{i\theta_1},\cdots,e^{i\theta_k})\mapsto e^{i(\theta_1 m_1+\theta_2m_2 +\cdots+\theta_k m_k)}.
\end{equation}
Given a representation $\rho$ of $G$, we may write the restriction $\rho|_T$ as a direct sum of irreducible (i.e. one-dimensional) representations as in Eq. (\ref{torus}).
\begin{definition}
\label{weight}
Given a representation $\rho$ of $G$, a weight of $\rho$ is a $k$-tuple $(m_1,\ldots,m_k)$ of integers corresponding to an irreducible representation of $T$ when $\rho$ is restricted to $T$ with the form of Eq. (\ref{torus}).
\end{definition}
The set of weights of a representation can be given a lexicographic order as \\$(m_1,\cdots,m_k) > (n_1,\cdots,n_k)$ if the first nonzero difference of $m_i - n_i$ is positive.  The largest weight with respect to this order is called the highest weight of the representation.  

Maximal tori are integral to the representation theory of Lie groups.  A maximal torus corresponds to a maximal abelian diagonalizable subalgebra $h$ (a Cartan subalgebra) of the Lie algebra $g$ corresponding to the Lie group $G$. The weights defined above are associated with the eigenvectors of the Cartan subalgebra when acting on $g$ by the representation under consideration. Since general representation theory is not our emphasis here, we refer the reader to extended discussions elsewhere, as in the latter half of \cite{F}.

Every element in a compact Lie group is conjugate to an element in a maximal torus, and any two maximal tori are themselves conjugate by an automorphism of $G$.  Hence, a character need only be defined on the maximal torus of a compact Lie group.  Note, however, that some elements of the maximal torus may be conjugate to one another.  If we let $N(T)=\{g\in G | gTg^{-1} = T\}$ denote the normalizer of $T$, then we define the Weyl group associated to the Lie group $G$ by 
\begin{equation} \nonumber
W = N(T)/T.
\end{equation}

The Weyl group of a Lie group acts on a maximal torus by conjugation.  Hence, one has an action on the set of weights by conjugation.  It is the basis of this key result. 

\begin{theorem}
If $p$ is an irreducible representation of $G$ with highest weight $\mu$, then the character is this representation is given by
\begin{equation}\nonumber
\chi_\rho(\theta) = \frac{\sum_{w\in W} \text{sgn}(w) e^{iw (\mu + \delta)\cdot \theta}}{\sum_{w\in W} \text{sgn}(w) e^{i w(\delta) \cdot \theta}}
\end{equation}
where $\theta = (\theta_1,\cdots, \theta_k)\in T$ is an element of a maximal torus, W is the Weyl group, and $\delta$ is the sum of the $k$ fundamental representations of the group.
\end{theorem}
For more information on the general theory of representations of Lie groups (and their associated Lie algebras), see, for example, \cite{Bour,knapp}. We now give a method for obtaining differential operators from representations of Lie groups, which is originally found in \cite{SW}. We state this general construction as a theorem. %Bourbaki, the classical groups, etc...

\begin{theorem}
%\begin{TGC}
Suppose $\rho_1: G\rightarrow GL(U)$ and $\rho_2:G\rightarrow GL(V)$ are irreducible representations of the Lie group $G$.  Suppose further that $U$ and $V$ are inner product spaces.  Suppose we have a function $f:U\rightarrow V$ which has continuous derivative.  Taking the gradient of the function $f$, we have
\begin{equation} \nonumber
\nabla f \in \text{Hom}(U,V) \cong U^* \otimes V,
\end{equation}
where $V^*$ is the dual vector space of $V$.  Using the metric on $V$ we identify $V$ with its dual by $v\mapsto (\cdot, v)$.  This gives an isomorphism $U\otimes V^* \rightarrow U\otimes V$, and so we may consider the gradient of $f$ to be in a map from $U$ into $U\otimes V$, that is,
\begin{equation} \nonumber
\nabla f \in \text{Hom}(U,V) \cong U^* \otimes V \cong U \otimes V.
\end{equation}

Denote by $U \circledcirc V$ the irreducible representation of $U\otimes V$ of highest weight.  This is known as the Cartan product of $\rho_1$ and $\rho_2$ \cite{ME}.  For each Lie group $G$, there exists a set of irreducible representations of $G$, known as the fundamental representations, such that an irreducible representation of $G$ can be written as a repeated Cartan product of the fundamental representations.  Using the inner products on $U$ and $V$, we may write
\begin{equation}
\nonumber
U\otimes V = \left (U \circledcirc V\right) \oplus \left(U \circledcirc V \right)^\perp.
\end{equation}

If we denote by $E$ and $E^\perp$ the orthogonal projections onto $U\circledcirc V$ and $(U\circledcirc V)^\perp$, then we define a differential operator $D$ associated to $\rho_1$ and $\rho_2$ by 
\begin{equation}
\label{gen_grad}
D = E^\perp \nabla.
\end{equation}

We obtain a series of equations by considering the system
\begin{equation}
\nonumber Df = 0.
\end{equation}
\end{theorem}

It is crucial to note this construction is in no way limited to functions defined on vector spaces.  Indeed, we may transfer this construction to a general Riemannian manifold that also has the structure of a principal $G$-bundle. Consider a manifold $M$.  A principal $G$-bundle over $M$ is a manifold $P$ together with a projection $p: P\rightarrow M$ such that $p^{-1}(x)$ is diffeomorphic to $G$ for all $x\in M$.  If $\pi_1(M)$ denotes the fundamental group of the manifold, then the universal covering space of the manifold is a principal $\pi_1(M)$-bundle.  A concrete example is the circle being covered by the line, in which case $p^{-1}(2\pi i \theta) = \{\theta +k \ | \ \theta \in [0,1), \ k\in\mathbb{Z}\}$ and diffeomorphism $p^{-1} (\theta) \rightarrow \mathbb{Z}$ is given by $x \mapsto x-\theta$.

Now if $P$ is a principal $G$-bundle over $M$, and $\rho_1$ and $\rho_2$ are representations of $G$ on $U$ and $V$, then we form the vector bundles $\mathscr{E}=P \times_{\rho_1} U$ by taking the product $P\times U$ and dividing out by the equivalence relation $(p\cdot g, v) ~(p, \rho_1(g)v)$.  Then we have two vector bundles $\mathscr{E}$ and $\mathscr{F}$ over $M$ corresponding to $\rho_1$ and $\rho_2$, and we form the tensor product $\mathscr{E} \otimes \mathscr{F}$ of these bundles.  Then the projection onto the subspace defined by the sub-vector bundle $U\circledcirc V$ (the Cartan product of $U$ and $V$) is known as a connection.

\subsection{The Spin Groups and Their Representations}
The special orthogonal group on $\R^n$, $SO(n)$, is the group of inner-product preserving linear transformations of determinant 1. It is well-known that $\pi_1(SO(n))=\mathbb{Z}_2$ for $n\geq 3$ \cite{AH}, so $SO(n)$ has a unique connected double covering group for $n\geq 3$.  We denote this group by $Spin(n)$. There is a concrete realization of $Spin(n)$ using the Clifford algebra $Cl_n$.  To this end, recall $vw +wv = -2(v,w)$ for $v,w \in \R^n$.  Hence, if $\|v\|=1$, we have that $vwv^{-1} = w-2(v,w)v$, which is precisely the formula for the reflection of $w$ across the 1-dimensional subspace defined by $v$.  By the Cartan-Diuedonne Theorem, $SO(n)$ is generated by even products of such reflections.  Hence, we define
\begin{equation} \nonumber
Spin(n) = \{v_1 v_2 \cdots v_{2k}\ |\  v_i \in S^{n-1} \ \forall i\}
\end{equation}

Writing $n=2k$ or $n=2k+1$ for when $n$ is even or odd, respectively we see that the maximal torus in $SO(n)$ is given by matrices of the form 
\begin{equation} \nonumber
\left(\begin{array}{cc}\text{cos} \theta_j & -\text{sin} \theta_j \\\text{sin}\theta_j & \text{cos} \theta_j\end{array}\right)
\end{equation}  
placed along the diagonal.  We denote by $T$ the preimage of the maximal torus of $SO(n)$ under the projection $Spin(n) \rightarrow SO(n)$.  The weights corresponding to $T$ are of the form $(m_1,\ldots, m_k)$ with all the $m_j$'s integral or half-integral.  The weights corresponding to the irreducible representations of the spin group may have half-integral values because it is a double cover of the special orthogonal group.  More precisely, if $T$ is a maximal torus in $SO(n)$ and $\pi:Spin(n) \rightarrow SO(n)$ is the projection map, then set $\tilde{T} = \pi^{-1}(T)$.  Then $T$ is an abelian group which is a connected double cover of $T$.  Hence, each irreducible representation of $\tilde{T}$ contributes "one half" an irreducible representation of $T$.  

The weights corresponding to the fundamental representations of $Spin(n)$ are \\ $(1,0,\cdots, 0)$, $(1,1, 0, \cdots, 0)$, $\cdots$, $(1,1,\cdots, 1, 0)$, and $(\frac{1}{2}, \frac{1}{2}, \cdots, \frac{1}{2})$ when $n=2k+1$, and $(1,0,\cdots, 0)$, $(1,1, 0, \cdots, 0)$, $\cdots$, $(1,1,\cdots, 1, 0,0)$, $(\frac{1}{2}, \frac{1}{2}, \cdots, \frac{1}{2})$, and $(\frac{1}{2}, \frac{1}{2}, \cdots, \frac{1}{2},-\frac{1}{2})$ when $n=2k$.  
The half-integral weight representations are known as the spin representations and are the only representations which are not lifted from a representation of $SO(n)$.  The representations of highest weight $(1,1,\cdots,1,0,\cdots, 0)$ with 1 repeated $r$ times are lifts of representations of $SO(n)$ and arise from the canonical representation of $SO(n)$ tensored with itself $r$ times, which are then restricted on the space of $r$-antisymmetric tensors.  For more information, see \cite{We}.  The spin representations can be realized by considering a minimal left ideal in the Clifford algebra $Cl_n$.  The spin group acts on this space by multiplication on the left.  This representation is irreducible if $n$ is odd and breaks up into two non-isomorphic representations that are dual to each other if $n$ is even.  

\subsection{The Orthogonal Dirac Operator as a Generalized Gradient}

We show the orthogonal Dirac operator is realizable by representations of the spin group.  

\begin{theorem}

Let $\rho_1$ be the representation of the spin group given by the standard representation of $SO(n)$ on $\R^n$
\begin{equation}\nonumber
 \rho_1:Spin(n) \rightarrow SO(n) \rightarrow GL(\R^n)
\end{equation}
and let $\rho_2$ be the spin representation on a minimal left ideal of $Cl_n$, which we denote $\mathbb{S}$.  Then the orthogonal Dirac operator is the differential operator given in the generalized gradient construction by $\R^n \circledcirc \mathbb{S}$ when $n=2k+1$ and by $\R^n \circledcirc \mathbb{S}^+$ when $n=2k$ and where $\mathbb{S}^+$ is the spin representation of weight $\left(\frac{1}{2},\frac{1}{2},\cdots,\frac{1}{2}\right)$.
\end{theorem}

\begin{proof} The theorem is proved in \cite{SW}, and we only outline the proof. One shows the kernel of the Dirac operator and kernel of the operator $D = E^\perp \nabla$ in Eq. (\ref{gen_grad}) for $E^\perp$ the orthogonal projection onto $(\R^n \circledcirc \mathbb{S})^\perp$ are the same, that is, these families of differential equations are equivalent. We consider the odd case $n=2k+1$. The even case $n=2k$ is similar, taking into account the character of the canonical representation is $\chi_\rho(\theta) =2 \sum_{j=1}^k \text{cos}(\theta_j)$ (not $\chi_\rho(\theta) = 2\sum_{j=1}^k \text{cos}(\theta_j)+1$ as in the odd case) and the fact the spin representation breaks up into two inequivalent, conjugate representations $\mathbb{S}=\mathbb{S}^+ \oplus \mathbb{S}^-$.

%\begin{equation}\nonumber
%\sum_{i=1}^n e_i \frac{\partial f}{\partial x_i} = 0
%\end{equation}

%is equivalent with the system 
%\begin{equation}\nonumber
%Df = E^\perp \nabla f =0
%\end{equation}
%where $E^\perp$ is the orthogonal projection onto $(\R^n \circledcirc \mathbb{S})^\perp$.  

%We shall first consider the odd case, $n=2k+1$.

In the odd case, the spin representation has highest weight $\left( \frac{1}{2}, \cdots, \frac{1}{2}\right)$.  Hence the Cartan product of the canonical representation and the spin representation has highest weight $\left(\frac{3}{2},\frac{1}{2},\cdots,\frac{1}{2}\right)$.  One defines an isomophism from $\mathbb{S}$ into $\R^n \otimes \mathbb{S}$ by
\begin{equation}\nonumber \eta(w) = \frac{1}{\sqrt{n}}(e_1w, \cdots, e_n w).
\end{equation}
One proceeds to show $\R^n \otimes \mathbb{S}= \left(\R^n \circledcirc \mathbb{S}\right)\oplus \eta(\mathbb{S})$. If this decomposition holds, then a solution $u$ to $Du = E^\perp \nabla u=0$ is $\mathbb{S}\text{-valued}$ and $\nabla u$ is $\R^n \otimes \mathbb{S}\text{-valued}$; moreover, $\nabla u$ is orthogonal to every element of $\eta(\mathbb{S})$. Noting, however, that as an endomorphism of $\R^n\otimes \mathbb{S}$, one has $e_j = - e_j^*$. Then one obtains.
\begin{equation}\nonumber
\sum_{i=1}^n \left(e_j \frac{\partial u}{\partial e_j},w\right)=0, \ \forall w\in \mathbb{S}, 
\end{equation}
 which says precisely that $u$ must be in the kernel of the orthogonal Dirac operator.
 
% We claim that $\R^n \otimes \mathbb{S}= \left(\R^n \circledcirc \mathbb{S}\right)\oplus \eta(\mathbb{S})$.  Now since $\frac{1}{n} \sum_{j=1}^n |e_j w|^2 = \frac{1}{n} \sum_{i=1}^n|w|^2 = |w|^2$, this is a unitary representation of $\mathbb{S}$.  Moreover, we have that
% \begin{equation}\nonumber \eta(\rho_2(g) w) = (\rho_1\otimes\rho_2)(g)(\eta(w)), \ w\in \mathbb{S}.
% \end{equation}
%Assume for a minute that $\R^n \otimes \mathbb{S}= \left(\R^n \circledcirc \mathbb{S}\right)\oplus \eta(\mathbb{S})$.  Now consider the equation $Du=0$, where $u$ has values in $\mathbb{S}$.  So $\nabla u$ has values in $\R^n \otimes \mathbb{S}$, and so the condition $Du = 0$ is equivalent with $\nabla u$ being orthogonal to $\eta(\mathbb{S})$.  This is precisely the statement that 
%\begin{equation}\nonumber
%\sum_{i=1}^n \left(\frac{\partial u}{\partial x_j}, e_j w\right) =0, \forall w \in \mathbb{S}.
%\end{equation}
%Notice, however, that as an endomorphism of $\R^n\otimes \mathbb{S}$, we have $e_j = - e_j^*$, hence the equations above become
%\begin{equation}\nonumber
%\sum_{i=1}^n \left(e_j \frac{\partial u}{\partial e_j},w\right)=0, \ \forall w\in \mathbb{S},
%\end{equation}
 %which says precisely that $u$ must be in the kernel of the orthogonal Dirac operator!
 
 The proof of the decomposition $\R^n \otimes \mathbb{S} = (\R^n \circledcirc \mathbb{S}) \oplus \eta(\mathbb{S})$ proceeds using character theory. Denoting by $\chi_\rho$ the character of the canonical representation and by $\chi_\mathbb{S}$ the character of the spin representation, we express the character of the representation on $\R^n \otimes \mathbb{S}$ as $\chi_\rho(\theta)\chi_\mathbb{S}(\theta)$. Decomposing this representation into irreducible representations and denoting by $\chi_m(\theta)$ the character of the irreducible representation of highest weight $m = (m_1, \cdots, m_k)$, we obtain
\begin{equation}\nonumber
\chi_\rho(\theta) \chi_\mathbb{S}(\theta) = \sum_m N(m)\chi_m(\theta).
\end{equation}
 Expanding the various $\chi_m$ using the Weyl character formula, we obtain
\begin{IEEEeqnarray}{C}
\nonumber \left(\sum_\sigma \text{sign}(\sigma)e^{i\sigma(\delta + (\frac{1}{2},\cdots,\frac{1}{2}))\cdot \theta}\right)\left(\sum_{j=1}^k e^{\pm i \theta_j}+1\right) = \sum_m N(m) \sum_\sigma \text{sign}(\sigma)e^{i\sigma(m+\delta) \cdot \theta},
\end{IEEEeqnarray}
where $\sigma$ runs over the Weyl group and $\delta = (k-\frac{1}{2}, k-\frac{3}{2}, \cdots, \frac{1}{2})$. Considering this equation, we see that any representation of strictly dominant highest weight must arise as either $m=\left(\frac{3}{2},\frac{1}{2},\cdots, \frac{1}{2}\right)$ or $m=(\frac{1}{2},\cdots,\frac{1}{2})$.  But the representation of highest weight $\left(\frac{3}{2},\frac{1}{2},\cdots, \frac{1}{2}\right)$ is the Cartan product of $\R^n$ and $\mathbb{S}$.  Hence, we have the desired decomposition.

%Choosing an explicit form of the maximal torus $T$ in $SO(n)$, we notice that $\chi_\rho(\theta) = 2\sum_{j=1}^k \text{cos}(\theta_j)+1$.  
%Hence, using the Weyl character formula, that is an irreducible character of weight $m$ takes may be written as 
%\begin{equation}\nonumber
%\chi_m(\theta) = \frac{\sum_\sigma \text{sign}(\sigma) e^{i\sigma(m+\delta)\cdot \theta}}{\sum_\sigma \text{sign}(\sigma) e^{i \sigma(\delta)\cdot \theta}}
%\end{equation}
 %Hence, we have
%\begin{IEEEeqnarray}{C}
%\nonumber \left(\sum_\sigma \text{sign}(\sigma)e^{i\sigma(\delta + (\frac{1}{2},\cdots,\frac{1}{2}))\cdot \theta}\right)\left(\sum_{j=1}^k e^{\pm i \theta_j}+1\right) = \sum_m N(m) \sum_\sigma \text{sign}(\sigma)e^{i\sigma(m+\delta) \cdot \theta}.
%\end{IEEEeqnarray}
\end{proof}

The transfer of the generalized gradient construction to a general Riemannian manifold that also has the structure of a principal $G$-bundle applies here, too. In the case $M$ is an oriented Riemannian manifold which has a principal $Spin(n)$-structure such that the projection onto $SO(n)$ on each fiber gives the $SO(n)$-structure associated to the Riemannian metric, then we may compose the total covariant derivative of the metric with the projection onto $TM \circledcirc \mathbb{S}M$, where $TM$ denotes the tangent bundle and $\mathbb{S}M$ denotes the spin bundle associated to the $Spin(n)$ structure. The resulting operator is the Atiyah-Singer Dirac operator on a Riemannian spin manifold \cite{ER}.  As this is the natural generalization of the orthogonal Dirac operator to a spin manifold, we argue this is the natural way to construct other Dirac-type operators as well.

\section{Hermitian Clifford Analysis}
A recent development in Clifford analysis is Hermitian Clifford analysis \cite{HCA1, HCA2}.  This branch of Clifford analysis attempts to refine standard orthogonal Clifford analysis by splitting the orthogonal Dirac operator into holomorphic and anti-holomorphic parts, each of which is invariant under an action of the unitary group.  While the original papers describing these operators claim they are generalized gradients in the sense of Stein and Weiss \cite{SW}, after reviewing the construction of these operators, we argue these operators are not natural from the point of view of representation theory, unlike the orthogonal Dirac operator.

\subsection{Complex structure and Clifford Algebras}
Let $V$ be an inner product space, $\text{dim } V_\R = 2n$.  Choose $J\in SO(V_\R)$ such that $J^2=-I$.  We will use the usual identification $SO(V_\R)\cong SO(2n)$.  Choose $J$ to be, say,
\begin{equation}\nonumber
\left(\begin{array}{cc}0 & -I \\I & 0\end{array}\right)
\end{equation}
where each block is an $n\times n$ matrix. 
Now the operators $\frac{1}{2}(I\pm i J)$ are mutually annihilating idempotents (i.e., projection operators) on $V_\C = V_\R \otimes_\R \C$.  Then we have the direct sum decomposition
\begin{equation}\nonumber
V_\C = W^+ \oplus W^-,
\end{equation}
where $W^\pm$ is the $\mp i $ eigenspace of $J$ acting on $V_\C$.  If we denote by $\rho$ the action of $SO(2n)$ on $V_\R$, then we may define an action on $V_\C$ by $\rho\otimes1$.  Then $g\in SO(2n)$ leaves $W^\pm$ invariant if and only if $gJ=Jg$.  Define the group $SO_J(2n)$ by 
\begin{equation} \nonumber
SO_J(2n) = \{g\in SO(2n) | gJ=Jg\}.
\end{equation}

\begin{proposition}
$SO_J(2n) \cong U(n)$.
\end{proposition} 
\begin{proof}
The relation $gJ = Jg$ gives 
\begin{eqnarray}\nonumber
\left(\begin{array}{cc}0 & -I \\I & 0\end{array}\right)\left(\begin{array}{cc}A & B \\C & D\end{array}\right)\left(\begin{array}{cc}0 & I \\-I & 0\end{array}\right) = \left(\begin{array}{cc}A & B \\C & D\end{array}\right)\\
\left(\begin{array}{cc}D & -C \\-B & A\end{array}\right) = \left(\begin{array}{cc}A & B \\C & D\end{array}\right)\nonumber
\end{eqnarray}
So $A=D$ and $B=-C$.  The relation $gg^\top=I$ gives 
\begin{eqnarray}\nonumber
AA^\top+BB^\top = I \\
AB^\top - BA^\top = 0 \nonumber
\end{eqnarray}
We claim the map $\phi: SO_J(2n) \rightarrow U(n)$, $\phi\left( \left(\begin{array}{cc}A & B \\-B & A\end{array}\right)\right) = A+iB$ is an isomorphism.  This map is clearly injective, as $\phi(g) = 0$ imples $A=0 \text{ and } B=0$.  Now suppose $u\in U(n)$.  Then set $A = \frac{u+u^*}{2}$ and $B = \frac{u-u^*}{2i}$, where $*$ denotes the conjugate transpose.  The relation $uu^*=I$ guarantees that 
\begin{equation}\nonumber
\left(\begin{array}{cc}A & B \\-B & A\end{array}\right)
\end{equation}
satisfies the orthogonality relations.  Hence, $\phi$ is an isomorphism.
\end{proof}

Now, suppose $\{e_j\}_{j=1}^{2n}$ is an orthonormal basis for $V_\R$.  Then define 
\begin{subequations}
\begin{align}
f_j &= \frac{1}{2}(I+iJ)e_j = \frac{1}{2}(e_j - i e_{n+j}) 
\label{Witt basis 1}\\
f^\dagger_j &= -\frac{1}{2} (1 - iJ)e_j=-\frac{1}{2}(e_j + i e_{n+j})\label{Witt basis 2}
\end{align}
\end{subequations}
to obtain the Witt basis of the complexified vector space $V_\C:= V_\R \otimes_\R \C$ \cite{Bour}.

\subsection{Hermitian Dirac Operators and Representation Theory}
Consider a smooth function $u$ defined on a subset of $V_\R$ taking values in $\mathbb{S}$.  Then the gradient is a member of $V_\R \otimes \mathbb{S}$.  We claim that $V_\R \otimes \mathbb{S} = V_1 \oplus V_2$, where
\begin{eqnarray}
\nonumber V_1 &=& \left\{\sum_{i=1}^n (e_j\otimes e_j\sigma + e_{n+j}\otimes e_{n+j} \sigma)\ : \ \sigma \in \mathbb{S}\right\}\\
\nonumber V_2 &=& \left\{\sum_{i=1}^n (e_j \otimes \sigma_j + e_{n+j} \otimes \sigma_{n+j} \ : \ \sum_{j=1}^n (e_j \sigma_j + e_{n+j}\sigma_{n+j})=0, \ \sigma_1, \ldots, \sigma_{2n} \in \mathbb{S}\right\}.
\end{eqnarray}
Clearly we have $V_1 \cap V_2 = 0$. We now define $v \in V_\R \otimes \mathbb{S}$ and $w \in \mathbb{S}$ by taking
\begin{eqnarray}\nonumber
v &=& \sum_{j=1}^n(e_j\otimes \sigma_j + e_{n+j} \otimes\sigma_{n+j})\\
%\end{equation}
%we set 
%\begin{equation}\nonumber
w &=& \sum_{j =1}^n (e_j \sigma_j + e_{n+j} \sigma_{n+j}) \nonumber.
\end{eqnarray}
Then we may write $v=v_1+v_2$, where 
\begin{eqnarray}\nonumber 
v_1 &=&-\frac{1}{2n} \sum_{j=1}^n(e_j\otimes e_j w + e_{n+j} \otimes e_{n+j} w)\\
v_2 &=& \sum_{j=1}^n (e_j \otimes \sigma_j + e_{n+j} \otimes \sigma_{n+j}) + \frac{1}{2n}\sum_{j=1}^n(e_j \otimes e_jw + e_{n+j} \otimes e_{n+j} w). \nonumber
\end{eqnarray}

Denote by $\pi_i$ the projection onto $V_i$, $i=1,2$.  Note this decomposition works perfectly well for the vector space $V_\C \otimes \mathbb{S}$, where the adjustment is simply to tensor $V_1$ and $V_2$ with $\C$.  It is important to note the subspace $V_1$ is isomorphic with the spinor space $\mathbb{S}$. Moreover, one should note the analogy with the orthogonal case above.  This isomorphism gives us invariance of $V_1$ under the action of $Spin(2n)$ on $V_\R\otimes \mathbb{S}$, which is the same action as in the orthogonal case.  We will now prove invariance of $V_2$ under the spin group action.  

To this end, suppose $v_2\in V_2$.  Then $v_2 = \sum_{j=1}^n (e_j \otimes \sigma_j + e_{n+j} \otimes \sigma_{n+j})$, where $\sum_{j=1}^n (e_j \sigma_j + e_{n+j} \sigma_{n+j})=0$.  Now if $s\in Spin(2n)$, then we have 
\begin{eqnarray}\nonumber s\cdot v_2 = \sum_{j=1}^n (s\cdot e_j \otimes s\cdot \sigma_j +s\cdot e_{n+j} \otimes s\cdot \sigma_{n+j})
= \sum_{j=1}^n (se_js^{-1} \otimes s\sigma_j + se_{n+j} s^{-1} \otimes s \sigma_{n+j}). \nonumber
\end{eqnarray}
Note that, since $v_2 \in V_2$, we have

\begin{equation}\nonumber
\sum_{j=1}^n (se_j s^{-1}s \sigma_j + se_{n+j}s^{-1}s\sigma_{n+j}) = s \sum_{j=1}^n (e_j\sigma_j +e_{n+j}\sigma_{n+j}) = 0,
\end{equation}
proving invariance of $V_2$ under the action of $Spin(2n)$.  

We are now ready to introduce the two Hermitian Dirac operators, defined by 
\begin{eqnarray}\nonumber
V_\C\overset{\nabla}{\rightarrow} V_\C\otimes \mathbb{S} \overset{\pi^+}{\rightarrow} W^+ \otimes \mathbb{S} \overset{\pi_1} {\rightarrow} V_1 \overset{\cong}{\rightarrow} \mathbb{S}\\
\nonumber V_\C\overset{\nabla}{\rightarrow} V_\C\otimes \mathbb{S} \overset{\pi^-}{\rightarrow} W^- \otimes \mathbb{S} \overset{\pi_1} {\rightarrow} V_1 \overset{\cong}{\rightarrow} \mathbb{S}.
\end{eqnarray}

Explicitly, for a function $F$ defined on $V_\C$ taking values in $\mathbb{S}$, we have
\begin{eqnarray}\nonumber
\nabla F \overset{\pi^+}{\rightarrow}2\sum_{j=1}^n f_j \otimes \partial_{ \overline{z}_j}F \overset{\pi_1}{\rightarrow} 2 \sum_{j=1}^nf_j\partial_{\overline{z}_j} F
\\ \nonumber
\nabla F \overset{\pi^-}{\rightarrow}2\sum_{j=1}^n f^{\dagger}_j \otimes \partial_{z_j}F \overset{\pi_1}{\rightarrow} 2 \sum_{j=1}^nf^{\dagger}_j\partial_{z_j} F,
\end{eqnarray}
where we have rewritten the partial differential operators in terms of the complex planes $z_j$ and their conjugates $\overline{z}_j$:
\begin{eqnarray}
\nonumber \partial_{\overline{z}_j} = \frac{1}{2}(\partial_{x_j} + i \partial_{y_j})\\
\nonumber \partial_{z_j} = \frac{1}{2} (\partial_{x_j} - i \partial_{y_j}).
\end{eqnarray}
We remind the reader that the dagger symbol denotes Hermitian conjugation, as in Eqs. (\ref{Witt basis 1}) and (\ref{Witt basis 2}). Now writing the operators  explicitly, we obtain 
\begin{eqnarray}
\nonumber \partial_Z = \sum_{j=1}^n f^{\dagger}_j\partial_{z_j}\\
\nonumber \partial_{Z^\dagger} = \sum_{j=1}^n f_j \partial_{\overline{z}_j},
\end{eqnarray}
which we call the Hermitian Dirac operators.  Note since we have generated these operators by projecting the gradient onto invariant subspaces, we have that these Hermitian Dirac operators are invariant under an action of $U(n)$.

\subsection{Failure of Analogy Between Orthogonal and Hermitian Dirac Operators}
One may have noticed the construction of the Hermitian Dirac operators was strongly analogous to the construction of the orthogonal Dirac operator using representation theory.  This analogy is intentional, but it is important to note that it is incomplete.  First, the analogy breaks down in that there are two Hermitian Dirac operators, though brief consideration tells one this is no great problem, since the analogy with complex analysis tells us that there should be a holomorphic and anti-holomorphic differential operator arising out of a unitary structure.  Indeed, one may see this in representation-theoretic terms as the space $V_\C$ decomposing into the $\mp i$ eigenspaces $W^\pm$, which are invariant under the action of $U(n)$ and give inequivalent, conjugate representations of $U(n)$.  

Second, and this is the crucial problem, the analogy fails when considering the restriction of the spin representation from the spin group to the double cover of the unitary group. Specifically, the representation
\begin{equation}
\nonumber U(n)\rightarrow GL(\mathbb{S})
\end{equation}
is not irreducible.  (Strictly speaking, neither is the spin representation, since we are in even dimensions.  This however, presents less of a problem as we have two inequivalent, conjugate representations).  For instance, the program Lie \cite{Lie} shows for $n=4$ that this representation breaks into the irreducible subspaces
\begin{equation} \nonumber
\mathbb{S} = \Lambda^1 V_\C \oplus \Lambda^3 V_\C \oplus \Lambda^5 V_\C \oplus \Lambda^7 V_\C
\end{equation}
Even restricting to the invariant subspaces of the spin representation gives either the irreducible subspaces $\Lambda^1 V_\C \oplus \Lambda^3 V_\C $ or $\Lambda^5 V_\C \oplus \Lambda^7 V_\C$ for the unitary group.  Attempting to place a complex structure on the vector space in question and force the Clifford algebra structures to respect the complex structure does not work in such an elementary way.

One objection to our claims is that the kernels of these Hermitian Dirac operators are important for obtaining the Howe dual pair $\mathfrak{sl}(1|2) \times U(n)$ that is important in representation theory \cite{Howe-dual-pair}. In this sense, such operators could be considered natural from a representation theory perspective. This may be correct as far as it goes, but the operators so constructed cannot claim to be the Hermitian analogue to the Stein-Weiss operators or to the Atiyah-Singer construction fundamental to spin geometry, since they do not project onto irreducible representations of the unitary group.

There are several ways to proceed for a Hermitian Clifford analysis modeled after the Stein-Weiss and Atiyah-Singer constructions.  One way is to attempt to construct operators using not the complexification of the Clifford algebra but instead the Clifford algebra over a complex vector space with respect to a Hermitian inner product (as opposed to a complexified orthogonal inner product).  This requires investigating the representation theory of the Hermitian Clifford algebra and its action with the unitary group or special unitary group, which are the natural groups respecting the inner product structure.

Another way is to consider a Cauchy-Riemann structure rather than a Hermitian structure.  One problem encountered in studying the Hermitian Dirac operators is the group under which the operators are invariant is relatively small because $U(n)$ is compact.  The orthogonal Dirac operator, however, is invariant over the larger group $SO(n+1,1)$, which is exactly the group of conformal transformations on the sphere.  There exists no analogous group for complex structures, but there exists a similar group for CR structures.  In particular, the sphere with its natural CR structure has $SU(n,1)$ as the group of CR-automorphisms.  Attempting to define Dirac operators on odd-dimensional spheres seems a natural way to proceed.  Moreover, since CR structures are less "rigid" than complex structures, there is more "room" for Dirac operators arising from CR structures to have a larger group under which it is invariant.  This avenue of research has begun \cite{RCK}.  

\section{Conclusion}
This work first reviewed the connection between the orthogonal Dirac operator and representations of the spin group, by projecting the gradient onto the orthogonal complement of the Cartan product of the orthogonal representation of the spin group and the spin representation of the spin group.  This construction transfers to a Riemannian spin manifold; we consider this the natural construction, since it gives the Atiyah-Singer Dirac operator on a spin manifold.  We then constructed the Hermitian Dirac operators.  Though these operators appear to arise as projections of the gradient onto invariant subspaces analogous to the orthogonal Dirac operator, the representations of the unitary group in question are not irreducible, suggesting the spaces under consideration are too ``big".  Hence, we argue, the Hermitian Dirac operators do not represent the natural generalization of the Dirac operators to spaces with complex structure.  Instead, one may pursue Dirac operators arising from a Clifford algebra associated to a Hermitian form. Seeking analogy with the conformal invariance of the orthogonal Dirac operator, one may also explore whether Dirac-type operators can be defined on manifolds with a Cauchy-Riemann structure, the early stages of which are being pursued \cite{RCK}.  Both possibilities present interesting avenues that should be further explored in future research.

\section*{Funding} S.S. acknowledges this material is based upon work supported by the University of Arkansas Honors College Bodenhamer Fellowship and a SILO/SURF Grant from the Arkansas Department of Higher Education. S.S. presented an earlier version of this work at the 9th International Conference on Clifford Algebras and their Applications at the Bauhaus University in Weimar, Germany in 2011. R.W. acknowledges this material is based upon work supported by the National Science Foundation Graduate Research Fellowship Program under Grant No. DGE-0957325 and the University of Arkansas Graduate School Distinguished Doctoral Fellowship in Mathematics and Physics.

\section*{Acknowledgement(s)} We thank two anonymous referees for critical comments that led to a more balanced manuscript.

%%%%%%%%%%%%%%     Reference     %%%%%%%%%%%%

\end{document}